\documentclass[12pt]{amsart}
\usepackage[T1]{fontenc}
\usepackage[utf8]{inputenc}
\usepackage{amsmath, amsthm, amsfonts, amssymb}
\usepackage{mathtools, mathrsfs, dsfont, relsize}
\usepackage{bm, bbm}
\usepackage{cancel}
\usepackage[normalem]{ulem}
\usepackage{graphicx}
\usepackage{xcolor}
\usepackage{enumerate, enumitem}
\usepackage{geometry}
\usepackage{tikz}
\usetikzlibrary{arrows.meta}
\usepackage{url, comment, verbatim}
\usepackage{todonotes}
\usepackage[most]{tcolorbox}
\usepackage[framemethod=tikz]{mdframed}
\usepackage{setspace}
\usepackage[pagewise]{lineno}
\graphicspath{{Images/}}

\usepackage{hyperref}
\hypersetup{
    colorlinks=true,
    linkcolor=blue,
    citecolor=blue,
    urlcolor=magenta
}

\usepackage{cleveref}  

\mathtoolsset{showonlyrefs=false} 
\numberwithin{equation}{section}

\theoremstyle{plain}
\newtheorem{theorem}{Theorem}

\newtheorem{corollary}[theorem]{Corollary}

\newtheorem{definition}[theorem]{Definition}
\newtheorem{remark}[theorem]{Remark}

\newcommand{\R}{\mathbb{R}}

\newcommand{\N}{\mathbb{N}}

\renewcommand{\S}{\mathbb{S}}

\newcommand{\supp}{\mathrm{supp}\,}

\author[F. Terzioglu and L. Yan]{%
Fatma Terzioglu$^{1,*}$ and Lili Yan$^1$\\[3pt]
\NoCaseChange{%
{\normalfont\small$^1$Department of Mathematics, North Carolina State University, Raleigh, NC, USA\\
fterzioglu@ncsu.edu, lyan6@ncsu.edu}}}
\thanks{$^*$Corresponding author}

\title[Range of Weighted Divergent Beam and Cone Integral Transforms]{Range Characterization of the Weighted Divergent Beam and Cone Integral Transforms}



\begin{document}
\begin{abstract}
We establish range characterizations, or data consistency conditions, for an integral transform that maps a function to its weighted integrals over conical surfaces in $\mathbb{R}^n$. We consider two different geometries for the cone vertices, which lead to mathematically distinct range conditions. We use the term \emph{conical Radon transform} when the vertex set is a bounded convex subset of $\mathbb{R}^n$ including support of the input function. The second geometry is motivated by Compton camera imaging where the vertex set represents planar detector locations and is disjoint from the support of the input function representing the radiation density. We refer to the associated transform as the \emph{Compton transform}.

Our approach is based on a factorization into the $k$-weighted divergent beam transform and the spherical section transform. In the bounded convex vertex geometry, the range of the divergent beam component is described by a
higher-order transport boundary-value problem, as studied by Derevtsov, Volkov, and Schuster \cite{Derevtsov2021}. In the planar detector geometry, we derive range conditions for the $k$-weighted divergent beam transform that generalize the planar cone-beam consistency conditions of Clackdoyle and Desbat \cite{ClackdoyleDesbat2013}. Combining these results with the range characterization of the spherical section transform yields complete range descriptions for both the $k$-weighted conical Radon transform and the $k$-weighted Compton transform.
\end{abstract}

\maketitle

\section{Introduction}\label{s:introduction}
Range characterizations, or equivalently data consistency
conditions, play an important role in integral geometry and
tomography. They identify the necessary and sufficient conditions under which a
given function can occur as the data of an integral transform; see, for example,
\cite{Agranovsky, Aguilar, Natt_old, Novikov, ClackdoyleDesbat2013}. In imaging
applications, these conditions distinguish physically and geometrically
consistent data from arbitrary functions on the data space. Consequently, they
provide consistency tests for measured data, guide the completion of missing or
truncated measurements, and can be used to detect or correct errors arising from
noise, detector miscalibration, or geometric imperfections in the acquisition
system; see \cite{Agranovsky, KuchCBMS, Patch2002,Lesaint2017} and the references therein.

In this paper, we establish range characterizations for the weighted conical Radon transform, which maps a function to its weighted integrals over conical surfaces in $\mathbb{R}^n$. This transform arises in the measurement models of $\gamma$-ray imaging modality called Compton camera imaging \cite{Terzioglu2018}.

Let $n\ge 2$, and let $\mathbb{S}^{n-1}\subset \mathbb{R}^n$ denote the unit
sphere, with surface measure $|\mathbb{S}^{n-1}|$. Given an apex, or vertex, $a\in A\subset \mathbb{R}^n$, a
direction $\beta\in \mathbb{S}^{n-1}$, and a (half-)opening angle
$\psi\in(0,\pi)$, the associated conical surface is
\[
\mathfrak{S}_{a,\beta,\psi}
:=
\{x\in\mathbb{R}^n:\ (x-a)\cdot\beta = |x-a|\cos\psi\}.
\]

\begin{definition}\label{CRT}
Let $k\in\mathbb{N}_0 = \N \cup \{0\}$. The $k$-weighted conical Radon transform of
$f\in \mathcal{C}_c^\infty(\mathbb{R}^n)$ is defined by
\begin{equation}\label{eq:CRT}
\mathscr{C}^k f(a,\beta,\psi)
:=
\int_{\mathfrak{S}_{a,\beta,\psi}}
f(x)\,|x-a|^{k-n+2}\,dS(x),
\end{equation}
where $dS(x)$ denotes surface measure on $\mathfrak{S}_{a,\beta,\psi}$.
\end{definition}

When the support of $f$ is disjoint from the vertex set $A$, the transform
\eqref{eq:CRT} is often called the \emph{Compton transform}. In Compton camera
imaging, the unknown function $f$ represents a radiation density, and the measured data are modeled by weighted integrals over conical surfaces whose vertices lie on the scattering detector. Thus the vertices represent detector bins and are assumed to lie outside the support of $f$:
\[
A\cap \supp f=\emptyset .
\]
The associated inverse problem is to recover the radiation density $f$ from these
Compton data. For more information on Compton camera imaging, see, for example,
\cite{Maxim2016, Terzioglu2018, Kim2024}.

Following \cite{Terzioglu2018}, we reserve the term \emph{Compton transform} for
this detector-radiation geometry and use \emph{conical Radon transform} for the
more general case in which the vertex set may intersect the support of $f$.
Although this more general transform does not directly model Compton camera
measurements, it has proven useful in their analysis; see
\cite{Terzioglu2015, KuchTer, Terzioglu2018}.

In Compton imaging, nontrivial weights are often included in the integral to model physical
effects such as the inverse-square law for gamma-ray emission and uniform
attenuation \cite{Munoz, Maxim2016}. Weighted variants of \eqref{eq:CRT}, many
motivated by Compton imaging, have been studied in
\cite{Smith, Maxim2009, KuchTer17, Palamodov, Haltmeier} and the references
therein. Further analytical properties of the Compton and conical Radon transforms were investigated in \cite{Terzioglu2019, Palamodov, Zhang}.

The range problem for conical Radon transforms has been addressed only in a few
settings. Moon \cite{Moon2016} and Baines \cite{Baines2021} studied range
characterizations for certain conical transforms. In particular, Baines obtained
a range characterization for an $n$-dimensional Compton transform defined on
cones with fixed opening angle, vertices in $\mathbb{R}^{n-1}$, and central axes
pointing in the north-pole direction. Based on this work, Jeon \cite{Jeon2025}
proposed a range description for an attenuated conical Radon transform with fixed
central axis and fixed opening angle. In \cite{Terzioglu2019}, we showed that the transform \eqref{eq:CRT} satisfies a differential equation that is independent of the power $k$; see Theorem~17 in \cite{Terzioglu2019}.

The present paper presents the complete range characterization of both the
$k$-weighted conical Radon transform and the $k$-weighted Compton transform. Our
analysis is based on the fact that these transforms factor as a composition of
the $k$-weighted divergent beam transform and the spherical section transform;
see Definitions~\ref{DBT} and~\ref{SST}. This reduces the range problem to two
separate questions: one for the $k$-weighted divergent beam transform and one for
the spherical section transform. The range characterization of the spherical
section transform was given in \cite{Terzioglu2025}, see also the references therein.

In view of this factorization, the cone vertices correspond to the source set of the associated divergent beam transform. This terminology is inherited from
cone-beam tomography, where the origins of the rays correspond to x-ray source
locations and the divergent beam transform is used to recover the attenuation
map \cite{Natt_old,NattWubb}. The inversion of the unweighted divergent beam transform has been extensively studied in cone-beam
tomography; see, for example,
\cite{Finch1985,Hamaker,SmithConeBeam,Gelfand-Goncharov,Tuy,Grangeat,Katsevich2002,Katsevich2004}
and the references therein. Reconstruction, unique continuation, and stability
properties of the $k$-weighted divergent beam transform were studied in
\cite{KuchTer17,Jathar2025}. Throughout the paper, we use \emph{source set} when referring to the divergent beam transform and \emph{vertex set} when referring to the conical Radon and Compton transforms. The two cases of cone integrals considered in this paper differ precisely in the geometry of this set.

For the conical Radon transform, we allow the vertex set to intersect the support of the input function $f$. Specifically, we take $A\subset\mathbb{R}^n$ to be bounded and convex, and assume
\[
\supp f\subset \bar{A},
\]
see Fig.~\ref{f:CRT_convex}. In this geometry, the range of the $k$-weighted divergent beam transform can be
described by a boundary-value problem for a higher-order transport equation. This
description follows from the work of Derevtsov, Volkov, and Schuster
\cite{Derevtsov2021}, who studied generalized attenuated ray transforms of order
$k$; the transform considered here corresponds to the unattenuated case.
Together with the range conditions for the spherical section transform, this
gives a range characterization for the $k$-weighted conical Radon transform.

For the Compton transform, the vertex set represents detector locations and is thus
disjoint from the support of the radiation density. We consider the planar
detector geometry
\[
A=\mathbb{R}^{n-1},
\qquad
\supp f\subset \mathbb{R}^n_+
=\mathbb{R}^{n-1}\times(0,\infty),
\]
so that
\[
\supp f\cap A=\emptyset;
\]
see Fig.~\ref{f:Compton_planar}. In this geometry, we establish range conditions
for the $n$-dimensional $k$-weighted divergent beam transform that extend the
cone-beam range conditions of Clackdoyle and Desbat
\cite{ClackdoyleDesbat2013}. Related data consistency conditions for other
source geometries include spherical source sets \cite{Finch1983}, line source
sets \cite{Sidky}, and circular source sets \cite{Clackdoyle2016}. Our result
generalizes the planar 3-dimensional unweighted case in
\cite{ClackdoyleDesbat2013} to higher dimensions and to the $k$-weighted setting.
Combining these conditions with the range characterization of the spherical
section transform yields the range characterization of the $k$-weighted Compton
transform.

The paper is organized as follows. In Section~\ref{s:preliminaries}, we define
the $k$-weighted divergent beam and spherical section transforms, derive their
relation to the conical Radon transform, and recall range characterization of the spherical section transform. In Section~\ref{s:rangeCRT}, we describe the range of the $k$-weighted divergent beam transform and the conical Radon transform in the
bounded convex vertex geometry. In Section~\ref{s:rangeCompton}, we establish
range conditions for the $k$-weighted divergent beam transform and the Compton
transform in the planar detector geometry. We conclude in
Section~\ref{s:conclusions}.

\section{Preliminaries}\label{s:preliminaries}
Using Dirac's distribution on the cone surface, and letting $s = \cos \psi$, we can equivalently define the $k$-weighted conical Radon transform \eqref{eq:CRT} as
\begin{align}\label{eq:ConeDeltaExplicit}
C^k f(a,\beta,s) 
&:= \mathscr{C}^k f(a,\beta, \cos^{-1} s) \\
&=\sqrt{1-s^2} \int_{\mathbb{R}^n}
f(x)\,\delta\!\big((x-a)\cdot\beta-|x-a|s\big)|x-a|^{k-n+2}\,dx. \nonumber
\end{align}
The $k$-weighted conical Radon transform can be expressed as a composition of the $k$-weighted divergent beam transform and the spherical section transform which are defined as follows.

\begin{definition}\label{def:DBT}
The $k$-weighted divergent beam transform of a function $f:\R^n \to \R$ is defined by
\begin{align}\label{DBT}
R^kf(a,v) = R^k_af(v) : = \int\limits_0^\infty f(a+r v)r^k \,dr,
\end{align}
where $a \in A \subset \R^n$ is the source of the beam directed along $v \in \S^{n-1}$. 
\end{definition}

\begin{definition}\label{def:SST}
The spherical section transform of a function $g:\S^{n-1} \to \R$ is defined by
\begin{align}\label{SST}
Sg(\beta,s) =\sqrt{1-s^2} \int\limits_{\S^{n-1}} g(v) \delta(v \cdot\beta-s) d v,
\end{align}
where $\beta\in \mathbb{S}^{n-1},\,s \in [-1,1]$, and $\delta$ is the one-dimensional Dirac's distribution.
\end{definition} 

Changing to spherical coordinates, with $x-a=rv$ in \eqref{eq:ConeDeltaExplicit}, we obtain
\begin{align} \label{Cone=SSTofDBT}
C^k f(a,\beta,s) = S(R^k_af)(\beta,s).
\end{align}

Therefore, the range of the conical Radon transform can be described by combining the range conditions for the $k$-weighted divergent beam transform and the spherical section transform. We establish a range characterization for the former in later sections. The necessary and sufficient conditions for a function to belong to the range of the spherical section transform are already known, and we recall them here.

\begin{theorem}[Range characterization of the spherical section transform] \label{t:rangeSST}
Let $n\ge 3$. A function $g \in \mathcal{C}^\infty(\mathbb{S}^{n-1} \times (-1,1))$ belongs to the range of the spherical section transform, that is $g = Sh$, for some function $h \in \mathcal{C}^\infty(\mathbb{S}^{n-1})$, if and only if \(g\) satisfies the following conditions:
\begin{enumerate}
\item \textbf{Evenness:} For any $(\beta,s) \in \mathbb{S}^{n-1} \times (-1,1)$,
$$g(-\beta,-s) =  g(\beta,s).$$
\item \textbf{PDE Condition:} For any $(\beta,s) \in \mathbb{S}^{n-1} \times (-1,1)$,
$$\Big[(1-s^2) \partial^2_s +(n-3)s \partial_s + \frac{n-2}{1-s^2} - \Delta_{\mathbb{S}^{n-1}}\Big] g(\beta,s) = 0,$$
where $\Delta_{\mathbb{S}^{n-1}}$ denotes the spherical Laplacian acting in the variable $\beta$.
\item \textbf{Limit Condition:} The normalized endpoint limits
\[
\lim_{s\to \pm 1}
\frac{g(\beta,s)}
{|\mathbb S^{n-2}|(1-s^2)^{(n-2)/2}},
\]
exist for every $\beta\in\mathbb S^{n-1}$ and define the smooth functions
$h(\pm\beta)$.
\end{enumerate}
\end{theorem}
\begin{proof}
The proof is a straightforward modification of the argument given in \cite[Appendix]{Terzioglu2025}.
\end{proof}

\section{\texorpdfstring{\boldmath Range characterization of the $k$-weighted conical Radon transform}{}}\label{s:rangeCRT}
Throughout this section, we consider the set of cone vertices $A\subset\mathbb{R}^n$ to be a bounded and convex subset of $\R^n$ and assume that $f\in \mathcal{C}_c^\infty(\mathbb{R}^n)$ with $\supp f\subset A$ (see fig. \ref{f:CRT_convex}). We first describe the range of the $k$-weighted divergent beam transform $R^k$. Then, using the identity \eqref{Cone=SSTofDBT}, we obtain complete range characterization of the conical Radon transform by combining the range conditions for $R^k$ with those for the spherical section transform $S$.

\begin{figure}[htbp]
\centering
\includegraphics[width=0.4\textwidth]{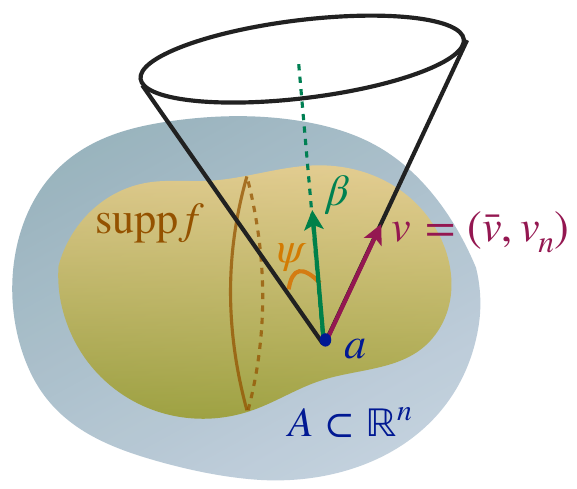}
\caption{Vertex set geometry for the conical Radon transform.}
\label{f:CRT_convex}
\end{figure}

\subsection{\texorpdfstring{\boldmath Range of the $k$-weighted divergent beam transform with source set containing $\supp f$}{}}
Suppose that $A\subset\mathbb{R}^n$ is a bounded open convex set with smooth boundary
$\partial A$. Let $n(a)$ denote the outward unit normal on $\partial A$. For
$v\in\mathbb{S}^{n-1}$, we set
\[
D_v:=v\cdot\nabla_a .
\]
We define the outflow boundary by
\[
\Gamma_+
:=
\{(a,v)\in \partial A\times \mathbb{S}^{n-1}:\ v\cdot n(a)>0\},
\]
and the forward exit time by
\[
\tau_+(a,v):=\inf\{t>0:\ a+t v\notin A\}, \quad (a,v)\in \bar{A}\times \mathbb{S}^{n-1}.
\]

For $k\in\mathbb{N}_0$ and $f\in \mathcal{C}_c^\infty(\mathbb{R}^n)$ with
$\supp f\subset A$, the integrand vanishes for $t\ge \tau_+(a,v)$, so
$R^k f(a,v)$ may equivalently be written as
\[
R^k f(a,v)
:=
\int_0^{\tau_+(a,v)} f(a+t v)\,t^k\,dt,
\qquad
(a,v)\in \bar{A}\times \mathbb{S}^{n-1}.
\]

For $k=0$, $R^0$ reduces to the usual ray transform \cite{Natt_old,NattWubb}, which is characterized as the unique  solution of a first-order transport equation with outflow boundary data \cite{Bal2009}. The weighted case
$k\ge 1$ can be viewed as an iterated antiderivative with respect to the
transport operator $D_v$, leading to a generalized transport equation of order $k+1$.

Following this approach, Derevtsov, Volkov, and Schuster \cite{Derevtsov2021}
studied attenuated ray transforms of order $k$. In particular, they showed that
the order-$k$ transform satisfies a transport equation of order $k+1$
\cite[Theorem~4]{Derevtsov2021}, and established uniqueness for the
corresponding higher-order attenuated boundary-value problem
\cite[Theorem~7]{Derevtsov2021}. The transform considered here is the
unattenuated analogue. In this case, both the higher-order equation and the
uniqueness statement follow from a direct characteristic argument, which we give
below for completeness.

\begin{theorem}[Range characterization of the $k$-weighted divergent beam transform]
\label{t:rangeRayBall}
Let $n\ge 2$ and $k\in \mathbb{N}_0$. Suppose that $A\subset\mathbb{R}^n$ is a bounded open convex set with smooth boundary and $f\in \mathcal{C}_c^\infty(\mathbb{R}^n)$ with
$\supp f\subset A$. For
$u\in \mathcal{C}^{k+1}(\bar{A}\times \mathbb{S}^{n-1})$, the following are equivalent:
\begin{equation}\label{eq:u=Rf}
u(a,v)=R^k f(a,v)
\qquad\text{for } (a,v)\in \bar{A}\times \mathbb{S}^{n-1},
\end{equation}
and
\begin{equation}\label{eq:BVP}
\begin{cases}
(D_v)^{k+1}u(a,v)=(-1)^{k+1}k!\,f(a),
& a\in A,\\[4pt]
(D_v)^j u(a,v)=0,
& (a,v)\in \Gamma_+,\quad j=0,1,\dots,k.
\end{cases}
\end{equation}
In particular, for each fixed $v\in\mathbb{S}^{n-1}$, the boundary-value problem
\eqref{eq:BVP} has at most one solution, and its unique solution is given by \eqref{eq:u=Rf}.
\end{theorem}

\begin{proof}
Let
\[
u_k(a,v):=R^k f(a,v)u_k(a,v)=\int_0^\infty f(a+t v)t^k\,dt.
\]
We first verify that $u_k$ satisfies \eqref{eq:BVP}. Since
\[
D_v f(a+tv)=\frac{d}{dt}f(a+tv),
\]
for $k=0$, differentiation under the integral sign gives
\[
D_v u_0(a,v)
=
\int_0^\infty \frac{d}{dt}f(a+t v)\,dt
=
-f(a),
\]
where the term at infinity vanishes because $f$ is compactly supported. For
$k\ge 1$, integration by parts gives
\[
D_v u_k(a,v)
=
\int_0^\infty \frac{d}{dt}f(a+t v)t^k\,dt
=
-k\int_0^\infty f(a+t v)t^{k-1}\,dt
=
-k u_{k-1}(a,v).
\]
Thus
\[
D_v u_k=-k u_{k-1},\qquad k\ge 1.
\]
Iterating this identity yields
\[
(D_v)^k u_k=(-1)^k k! u_0.
\]
Applying $D_v$ once more and using $D_vu_0=-f$ gives
\[
(D_v)^{k+1}u_k
=
(-1)^k k! D_vu_0
=
(-1)^{k+1}k! f.
\]

It remains to verify the boundary conditions. If $(a,v)\in\Gamma_+$, then
$\tau_+(a,v)=0$, and hence $u_k(a,v)=0$. Moreover, by convexity of $A$, the
forward ray starting at such a boundary point remains outside $A$. Therefore
\[
u_k(a+r v,v)=0,\qquad r\ge 0,
\]
and differentiating along the characteristic gives
\[
(D_v)^j u_k(a,v)=0,
\qquad
(a,v)\in\Gamma_+,\quad j=0,1,\dots,k.
\]
Hence $u_k=R^k f$ satisfies \eqref{eq:BVP}.

Conversely, suppose that $u$ satisfies \eqref{eq:BVP}. Since $u_k=R^k f$
also satisfies \eqref{eq:BVP}, the difference
\[
w:=u-u_k
\]
satisfies the homogeneous problem
\[
(D_v)^{k+1}w=0 \quad \text{in } A,
\qquad
(D_v)^j w=0 \quad \text{on } \Gamma_+,\quad j=0,1,\dots,k.
\]
We now prove that $w\equiv 0$. Fix $a\in A$ and
$v\in\mathbb{S}^{n-1}$, and define
\[
\varphi(t):=w(a+t v,v),
\qquad
0\le t\le \tau_+(a,v).
\]
Then
\[
\varphi^{(j)}(t)=(D_v)^j w(a+t v,v),
\qquad j=0,1,\dots,k+1.
\]
Since $(D_v)^{k+1}w=0$, the function $\varphi$ is a polynomial of degree at
most $k$. Let
\[
a_+:=a+\tau_+(a,v)v\in\partial A
\]
be the forward exit point. Since $A$ is convex, $(a_+,v)\in\Gamma_+$. Therefore
the boundary conditions imply
\[
\varphi^{(j)}(\tau_+(a,v))
=
(D_v)^j w(a_+,v)
=
0,
\qquad j=0,1,\dots,k.
\]
A polynomial of degree at most $k$ whose derivatives of orders $0,\dots,k$
vanish at one point is identically zero. Hence $\varphi\equiv 0$, and in
particular $w(a,v)=0$. Since $a$ and $v$ were arbitrary, $w\equiv 0$ in
$A\times\mathbb{S}^{n-1}$. Therefore
\[
u(a,v)=R^k f(a,v),
\]
as claimed.

Finally, the uniqueness statement follows by applying the same argument to
the difference of two solutions.
\end{proof}

\subsection{\texorpdfstring{\boldmath Range of the conical Radon transform with cone vertices containing $\supp f$}{}}
By combining the range characterizations of the spherical section transform and the $k$-weighted divergent beam transform, we arrive at the range characterization for the conical Radon transform.

\begin{theorem}[Range characterization of the weighted conical Radon transform]
\label{t:rangeCRT}
Let \(n\ge 3\) and \(k\in \mathbb N_0\). Suppose that
$A\subset\mathbb{R}^n$ is a bounded open convex set with smooth boundary, and
\[
g \in \mathcal{C}^\infty(\bar{A}\times \mathbb S^{n-1}\times (-1,1)).
\]
Then $g$ belongs to the range of the $k$-weighted conical Radon transform, that is,
$g=C^k f$ for some $f\in \mathcal{C}_c^\infty(A)$, if and only if $g$ satisfies the following conditions:
\begin{enumerate}
\item[(i)] \textbf{Evenness:} For every \(a\in \bar{A}\),
\[
g(a,-\beta,-s)=g(a,\beta,s),
\qquad
(\beta,s)\in \mathbb S^{n-1}\times (-1,1).
\]

\item[(ii)] \textbf{PDE Condition:} For every \(a\in \bar{A}\),
\[
\Big[(1-s^2)\partial_s^2+(n-3)s\partial_s+\frac{n-2}{1-s^2}-\Delta_{\mathbb S^{n-1}}\Big]
g(a,\beta,s)=0,
\qquad
(\beta,s)\in \mathbb S^{n-1}\times (-1,1),
\]
where $\Delta_{\mathbb S^{n-1}}$ denotes the spherical Laplacian.

\item[(iii)] \textbf{Limit Condition:} The limit
\[
u(a,\beta):=\lim_{s\to 1}
\frac{g(a,\beta,s)}{|\mathbb S^{n-2}|(1-s^2)^{(n-2)/2}},
\]
exists for every \((a,\beta)\in \bar{A}\times \mathbb S^{n-1}\), and defines a smooth function
\[
u\in \mathcal{C}^\infty(\bar{A}\times \mathbb S^{n-1}).
\]
Furthermore, $u$ satisfies the following conditions:
\begin{enumerate}
\item \textbf{Outflow boundary conditions:}
\[
(D_v)^j u(a,v)=0,
\qquad
(a,v)\in \Gamma_+,\quad j=0,1,\dots,k.
\]

\item \textbf{Directional Independence:} The function $(D_v)^{k+1}u(a,v)$ is independent of \(v\in \mathbb S^{n-1}\) for each \(a\in A\).

\item \textbf{Compact support in the source set:} There exists a compact set
\(K\Subset A\) such that
\[
(D_v)^{k+1}u(a,v)=0,
\qquad
a\in A\setminus K,\quad v\in \mathbb S^{n-1}.
\]
\end{enumerate}
\end{enumerate}

In particular, the function
\[
f(a):=\frac{(-1)^{k+1}}{k!}(D_v)^{k+1}u(a,v),
\]
satisfies
\[
u(a,v)=R_a^k f(v),
\qquad\text{and}\qquad
g(a,\beta,s)=S(R_a^k f)(\beta,s)=C^k f(a,\beta,s).
\]
Moreover, \(f\) is uniquely determined by \(g\).
\end{theorem}

\begin{proof}
We first prove necessity. Suppose that there exists \(f\in \mathcal{C}_c^\infty(A)\) such that
\[
g(a,\beta,s)=C^k f(a,\beta,s)=S(R_a^k f)(\beta,s).
\]
Define
\[
u(a,v):=R_a^k f(v).
\]
Then, for each fixed \(a\in \bar{A}\), the function \((\beta,s)\mapsto g(a,\beta,s)\)
is the spherical section transform of \(u(a,\cdot)\). Hence,
conditions \emph{(i)} and \emph{(ii)} immediately follow from \Cref{t:rangeSST}. Also,
\[
\lim_{s\to 1}
\frac{g(a,\beta,s)}{|\mathbb S^{n-2}|(1-s^2)^{(n-2)/2}} = u(a,\beta).
\]

Since \(u(a,v)=R_a^k f(v)\), \Cref{t:rangeRayBall} guarantees that the boundary conditions hold:
\[
(D_v)^j u(a,v)=0,
\qquad
(a,v)\in \Gamma_+,\quad j=0,1,\dots,k,
\]
and the function $u$ satisfies
\[
(D_v)^{k+1}u(a,v)=(-1)^{k+1}k!\,f(a),
\qquad
a\in A,
\]
the right-hand side of which is independent of \(v\). Finally, since
\(f\in \mathcal{C}_c^\infty(A)\), there exists a compact set \(K\Subset A\) such that
\[
\supp f\subset K.
\]
Therefore
\[
f(a)=0,\qquad a\in A\setminus K,
\]
which immediately implies
\[
(D_v)^{k+1}u(a,v)=0,
\qquad
a\in A\setminus K,\quad v\in \mathbb S^{n-1}.
\]
Thus all requirements in \emph{(iii)} are satisfied.

We now prove sufficiency. Suppose that \(g\) satisfies conditions \emph{(i)}--\emph{(iii)}. We shall explicitly construct the function $f$ such that $g = C^kf$.

We define $u$ by the given limit:
\[
u(a,\beta):=\lim_{s\to 1}
\frac{g(a,\beta,s)}{|\mathbb S^{n-2}|(1-s^2)^{(n-2)/2}}.
\]
Since \((D_v)^{k+1}u(a,v)\) is independent of \(v\), we may define a function
\[
f(a):=\frac{(-1)^{k+1}}{k!}(D_v)^{k+1}u(a,v),
\qquad a\in A.
\]

By definition, $f$ is smooth and we have
\[
(D_v)^{k+1}u(a,v)=(-1)^{k+1}k!\,f(a),
\qquad a\in A,\ v\in \mathbb S^{n-1}.
\]

The support condition in \emph{(iii)} gives a compact set \(K\Subset A\) such that
\[
(D_v)^{k+1}u(a,v)=0,
\qquad
a\in A\setminus K,\quad v\in\mathbb S^{n-1}.
\]
Therefore
\[
f(a)=\frac{(-1)^{k+1}}{k!}(D_v)^{k+1}u(a,v)=0,
\qquad
a\in A\setminus K.
\]
Hence \(\supp f\subset K\Subset A\), and so \(f\in C_c^\infty(A)\).

Together with the boundary conditions in \emph{(iii)}, this shows that \(u\) satisfies the boundary value problem \eqref{eq:BVP}.

Hence, by \Cref{t:rangeRayBall},
\[
u(a,v)=R_a^k f(v).
\]

It remains to show that $g(a,\beta,s)=S(u(a,\cdot))(\beta,s)$. For each fixed \(a\in \bar{A}\), condition \emph{(iii)} implies that
\[
\lim_{s\to 1}\frac{g(a,\beta,s)}{|\mathbb S^{n-2}|(1-s^2)^{(n-2)/2}}
= u(a,\beta).
\]
By the evenness condition \emph{(i)}, we have
\[
\lim_{s\to -1}\frac{g(a,\beta,s)}{|\mathbb S^{n-2}|(1-s^2)^{(n-2)/2}}
=
\lim_{t\to 1}\frac{g(a,-\beta,t)}{|\mathbb S^{n-2}|(1-t^2)^{(n-2)/2}}
= u(a,-\beta).
\]
Together with conditions \emph{(i)} and \emph{(ii)}, in view of \Cref{t:rangeSST}, this implies that
\[
g(a,\beta,s)=S(u(a,\cdot))(\beta,s).
\]
Substituting \(u(a,v)=R_a^k f(v)\), we obtain
\[
g(a,\beta,s)=S(R_a^k f)(\beta,s)=C^k f(a,\beta,s).
\]

The uniqueness of \(f\) follows from the explicit reconstruction formula
\[
f(a)=\frac{(-1)^{k+1}}{k!}(D_v)^{k+1}u(a,v),
\]
since \(u\) is uniquely determined by \(g\).
\end{proof}

\section{Range characterization of the Compton transform for infinite planar detectors}
\label{s:rangeCompton}
In the previous section, we considered the conical Radon transform where the cone vertex set is a bounded convex subset of $\mathbb{R}^n$ containing the support of the unknown function. We now turn to the second vertex geometry considered in this paper, motivated by Compton camera imaging. Here the cone vertices represent detector locations and are disjoint from the support of the radiation density. Following the terminology introduced in the introduction, we refer to the corresponding cone integral transform as the \emph{Compton transform}.

Throughout this section, we set
\[
A=\mathbb{R}^{n-1}
\]
and identify \(A\) with the detector hyperplane \(A\times\{0\}\subset
\mathbb{R}^n\). We assume
\[
f\in \mathcal{C}_c^\infty(\mathbb{R}^n_+),
\qquad
\mathbb{R}^n_+:=\mathbb{R}^{n-1}\times(0,\infty),
\]
so that
\[
\supp f\cap (A\times\{0\})=\emptyset;
\]
see Fig.~\ref{f:Compton_planar}.

\begin{figure}[htbp]
\centering
\includegraphics[width=0.45\textwidth]{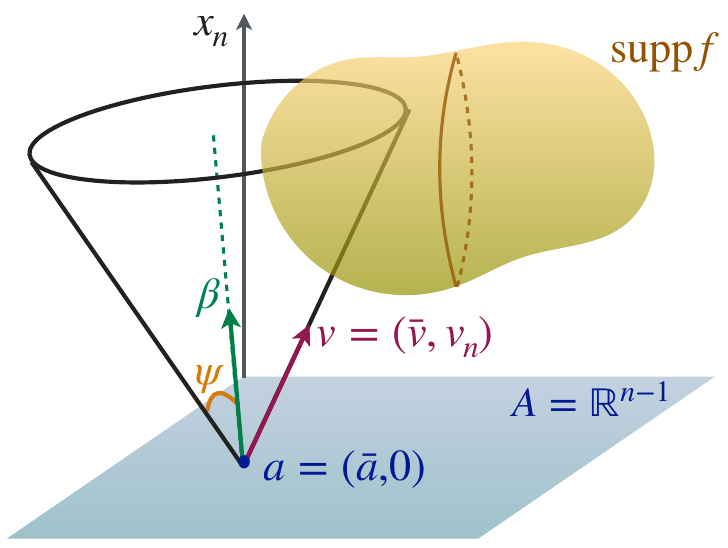}
\caption{Vertex set geometry for the Compton transform.}
\label{f:Compton_planar}
\end{figure}

As in the previous section, we begin with the range characterization of the
corresponding $k$-weighted divergent beam transform. In the present geometry,
the source set of the divergent beam transform coincides with the planar vertex
set \(A\).

\subsection{\texorpdfstring{\boldmath Range of the $k$-weighted divergent beam transform with source set $A=\mathbb{R}^{n-1}$ disjoint from $\supp f$}{}}

For \(\bar a\in A\), the corresponding source point is
\[
a=(\bar a,0)\in\mathbb{R}^n,
\qquad
\bar a=(a_1,\dots,a_{n-1}).
\]
We consider rays having directions in the upper hemisphere
\[
\mathbb{S}^{n-1}_+
:=
\{v=(\bar v,v_n)\in\mathbb{S}^{n-1}: v_n>0\}.
\]
Throughout this section, the bar notation is reserved for elements of
\(\mathbb{R}^{n-1}\).

For \(f\in C_c^\infty(\mathbb{R}^n_+)\), we recall that the \(k\)-weighted
divergent beam transform with source point \(a=(\bar a,0)\) and direction
\(v\in\mathbb{S}^{n-1}_+\) is
\[
R^k f(a,v)=R_a^k f(v)
:=
\int_0^\infty f(a+r v)r^k\,dr.
\]

\begin{theorem}[Range characterization of the $k$-weighted divergent beam transform]\label{t:rangeRayPlane} 
Let $n\ge 2$ and $k\in \mathbb{N}_0$.
A function
\[
u\in \mathcal{C}^\infty(\mathbb{R}^{n-1}\times \mathbb{S}^{n-1}_+),
\]
belongs to the range of the $k$-weighted divergent beam transform, that is, there exists $f\in \mathcal{C}_c^\infty(\mathbb{R}^n_+)$ such that $u(\bar{a}, v) = R^k f(a, v)$ for $a = (\bar{a}, 0) \in \mathbb{R}^{n-1} \times \{0\}$, if and only if the scaled projective data 
\begin{align}\label{scaled_proj_data}
w(\bar{a},\bar{p})  = v_n^{k+1} u (\bar{a}, v), \quad \bar{p} = \frac{\bar{v}}{v_n} \in \mathbb{R}^{n-1},
\end{align}
and its Fourier transform with respect to $\bar{p}$,
\[
W(\bar{a}, \bar{\xi}) = \int_{\mathbb{R}^{n-1}} e^{-i\bar{p}\cdot \bar{\xi}} w(\bar{a}, \bar{p}) \,d\bar{p},
\]
satisfy the following conditions:

\begin{enumerate}
\item \textbf{Compact Support:} For every fixed $\bar{a} \in \mathbb{R}^{n-1}$, the function $w(\bar{a}, \cdot)$ has compact support.
\item \textbf{PDE Condition:} There exists a function $H(\sigma, \bar{\xi})$ such that $W(\bar{a}, \bar{\xi}) = H(\bar{a}\cdot\bar{\xi}, \bar{\xi})$. Equivalently, 
\[
\bar{b} \cdot \nabla_{\bar{a}} W(\bar{a}, \bar{\xi}) = 0, \quad \text{for all } \bar{b}\in \mathbb{R}^{n-1} \text{ such that } \bar{b} \cdot \bar{\xi} = 0,
\]
which implies that $W$ is constant in $\bar a$ along affine hyperplanes orthogonal to
$\bar\xi \neq \mathbf{0}$.
\item \textbf{Support and Smoothness:} The inverse Fourier transform of the function $H$ in condition (ii) with respect to both variables, defined as
\[
h(t, \bar{p}) := \frac{1}{(2\pi)^n} \int_{\mathbb{R}^{n-1}} \int_{\mathbb{R}} e^{i(t\sigma + \bar{p}\cdot\bar{\xi})} H(\sigma, \bar{\xi}) \, d\sigma \, d\bar{\xi},
\]
is a smooth function on $\mathbb{R} \times \mathbb{R}^{n-1}$. Furthermore, there exist constants $0 < m_0 < M_0 < \infty$ and $R > 0$ such that the support of $h$ satisfies
\[
\supp h \subset \left\{ (t, \bar{p}) \in \mathbb{R} \times \mathbb{R}^{n-1} : t \in [-M_0, -m_0] \text{ and } |\bar{p}| \le -R t \right\}.
\]
\end{enumerate}
\end{theorem}

\begin{proof}
We first prove necessity. Assume that there exists $f\in \mathcal{C}_c^\infty(\mathbb{R}^n_+)$ such that $u(\bar{a}, v) = R^k f(a, v)$ for sources $a = (\bar{a}, 0) \in \R^{n-1} \times \{0\}$.

Condition \emph{(i)} is satisfied since $\supp f$ is compact. Indeed, for any fixed $\bar{a} \in \R^{n-1}$, the divergent beam transform $R^k f(a, v) = u(\bar{a}, v)$ vanishes for all directions $v$ such that the ray starting at $a$ and traveling in direction $v$ misses the support of $f$, and thus $w(\bar{a}, \cdot)$ has compact support.

Now we show that condition \emph{(ii)} holds.
For $v_n > 0$, we substitute the projective variable $\bar{p} = \bar{v}/v_n$ and $x_n = v_n r$ into the transform definition to extract the scaled projective data:
\begin{equation*}
w(\bar{a}, \bar{p}) := v_n^{k+1} u(\bar{a}, v) = \int_0^\infty f(\bar{a} + x_n \bar{p}, x_n) x_n^k \, dx_n.
\end{equation*}
We then take the Fourier transform of $w$ with respect to $\bar{p}$:
\begin{equation*}
W(\bar{a}, \bar{\xi}) = \int_{\mathbb{R}^{n-1}} e^{-i\bar{p}\cdot \bar{\xi}} \left[ \int_0^\infty f(\bar{a} + x_n \bar{p}, x_n) x_n^k \, dx_n \right] d\bar{p}.
\end{equation*}
Letting $\bar{x} = \bar{a} + x_n \bar{p}$, and using Fubini's theorem, we have
\begin{align*}
W(\bar{a}, \bar{\xi}) &= \int_0^\infty x_n^k \left[ \int_{\mathbb{R}^{n-1}} e^{-i \left(\frac{\bar{x} - \bar{a}}{x_n}\right) \cdot \bar{\xi}} f(\bar{x}, x_n) x_n^{1-n} \, d\bar{x} \right] dx_n \\
&= \int_0^\infty x_n^{k-n+1} e^{i \frac{\bar{a}\cdot\bar{\xi}}{x_n}} \left[ \int_{\mathbb{R}^{n-1}} f(\bar{x}, x_n) e^{-i \bar{x} \cdot \left(\frac{\bar{\xi}}{x_n}\right)} \, d\bar{x} \right] dx_n.
\end{align*}
Let $F(\bar{\eta}, x_n)$ denote the Fourier transform of $f(\bar{x},x_n)$ with respect to $\bar{x}$. 
Then,
\begin{equation*}
W(\bar{a}, \bar{\xi}) = \int_0^\infty x_n^{k-n+1} e^{i \frac{\bar{a}\cdot\bar{\xi}}{x_n}} F\left(\frac{\bar{\xi}}{x_n}, x_n\right) \, dx_n.
\end{equation*}
Next we change variables by letting $\lambda = 1/x_n$. As $x_n \in (0, \infty)$, we have $\lambda \in (0, \infty)$ and $dx_n = -\frac{1}{\lambda^2} d\lambda$. Thus,
\begin{align*}
W(\bar{a}, \bar{\xi}) &= \int_0^\infty \left(\frac{1}{\lambda}\right)^{k-n+1} e^{i \lambda (\bar{a}\cdot\bar{\xi})} F\left(\lambda\bar{\xi}, \frac{1}{\lambda}\right) \frac{1}{\lambda^2} \, d\lambda \\
&= \int_0^\infty \lambda^{n-k-3} e^{i \lambda (\bar{a}\cdot\bar{\xi})} F\left(\lambda\bar{\xi}, \frac{1}{\lambda}\right) \, d\lambda.
\end{align*}
Notice that the right-hand side depends on $\bar{a}$ only through the dot product $\sigma = \bar{a}\cdot\bar{\xi}$. Therefore, the condition \emph{(ii)} is satisfied with
\begin{equation*}
H(\sigma, \bar{\xi}) := \int_0^\infty \lambda^{n-k-3} e^{i \lambda \sigma} F\left(\lambda\bar{\xi}, \frac{1}{\lambda}\right) \, d\lambda.
\end{equation*}

Lastly, we verify condition \emph{(iii)}.
To compute $h(t, \bar{p})$, we first compute the inverse Fourier transform of $H(\sigma, \bar{\xi})$ with respect to $\sigma$, denoted $\tilde{H}(t, \bar{\xi})$. By definition of $H$ and the Fubini's theorem, we have
\begin{align*}
\tilde{H}(t, \bar{\xi}) &= \frac{1}{2\pi} \int_{\mathbb{R}} e^{i\sigma t} \left[ \int_0^\infty \lambda^{n-k-3} e^{i \lambda \sigma} F\left(\lambda\bar{\xi}, \frac{1}{\lambda}\right) \, d\lambda \right] d\sigma \\
&= \int_0^\infty \lambda^{n-k-3} F\left(\lambda\bar{\xi}, \frac{1}{\lambda}\right) \delta(t + \lambda) \, d\lambda.
\end{align*}
By definition of the Dirac delta distribution, we obtain

\begin{equation*}
\tilde{H}(t, \bar{\xi}) = 
\begin{cases} 
(-t)^{n-k-3} F\left(-t\bar{\xi}, -\frac{1}{t}\right), & \text{if } t < 0 \\
0, & \text{if } t \ge 0.
\end{cases}
\end{equation*}
Now we compute $h(t, \bar{p})$ by taking the inverse Fourier transform of $\tilde{H}(t, \bar{\xi})$ with respect to $\bar{\xi}$. For $t < 0$, we have:
\begin{equation*}
h(t, \bar{p}) = \frac{1}{(2\pi)^{n-1}} \int_{\mathbb{R}^{n-1}} e^{i\bar{p}\cdot\bar{\xi}} (-t)^{n-k-3} F\left(-t\bar{\xi}, -\frac{1}{t}\right) \, d\bar{\xi}.
\end{equation*}
Applying the change of variables $\bar{\eta} = -t\bar{\xi}$, we get 
\begin{align*}
h(t, \bar{p}) &= (-t)^{-k-2} \frac{1}{(2\pi)^{n-1}} \int_{\mathbb{R}^{n-1}} e^{i\left(-\frac{\bar{p}}{t}\right)\cdot\bar{\eta}} F\left(\bar{\eta}, -\frac{1}{t}\right) \, d\bar{\eta} \\
&= (-t)^{-k-2} f\left(-\frac{\bar{p}}{t}, -\frac{1}{t}\right).
\end{align*}
Therefore,
\begin{equation*}
h(t, \bar{p}) = 
\begin{cases} 
(-t)^{-k-2} f\left(-\frac{\bar{p}}{t}, -\frac{1}{t}\right), & \text{if } t < 0 \\
0, & \text{if } t \ge 0.
\end{cases}
\end{equation*}

Recall that $f \in \mathcal{C}_c^\infty(\mathbb{R}^n_+)$. Thus, $f$ is smooth, and there exist constants $M_0 > m_0 > 0$ and $R > 0$ such that $f(\bar{x}, x_n) = 0$ whenever $x_n \notin [1/M_0, 1/m_0]$ or $|\bar{x}| > R$. Consequently, $h(t, \bar{p})$ vanishes unless $-1/t \in [1/M_0, 1/m_0]$, which implies $t \in [-M_0, -m_0]$. Furthermore, $h(t, \bar{p})$ vanishes unless $|-\bar{p}/t| \le R$, which exactly means $|\bar{p}| \le -R t$. Since $f$ is smooth and the singularity at $t=0$ is strictly avoided, $h$ is smooth everywhere. This establishes necessity.

We now prove sufficiency.
Assume that we are given a smooth function $u$ on $\mathbb{R}^{n-1}\times \mathbb{S}^{n-1}_+$ satisfying conditions \emph{(i)}--\emph{(iii)}. We will explicitly construct a function $f$ so that $R^kf=u$.

By condition \emph{(iii)}, the function $h(t, \bar{p})$ is smooth and strictly supported in the region where $t \in [-M_0, -m_0]$ and $|\bar{p}| \le -R t$. We define the function $f(\bar{x}, x_n)$ directly from $h$ via the transformation derived in necessity:
\begin{equation}\label{f_def_from_h}
f(\bar{x}, x_n) := 
\begin{cases}
x_n^{-k-2} h\left(-\frac{1}{x_n}, \frac{\bar{x}}{x_n}\right), & \text{for } x_n > 0 \\
0, & \text{for } x_n \le 0.
\end{cases}
\end{equation}
Since $h(t, \bar{p}) = 0$ for $t \notin [-M_0, -m_0]$, it follows that $f(\bar{x}, x_n) = 0$ whenever $-1/x_n \notin [-M_0, -m_0]$, meaning $x_n \notin [1/M_0, 1/m_0]$. Because $f$ is identically zero in a neighborhood of $x_n = 0$, the composition is infinitely differentiable, ensuring $f \in \mathcal{C}^\infty(\mathbb{R}^n)$. Furthermore, since $h(t, \bar{p}) = 0$ whenever $|\bar{p}| > -R t$, it follows that $f(\bar{x}, x_n) = 0$ whenever $|\bar{x}/x_n| > -R(-1/x_n) = R/x_n$, which simplifies to $|\bar{x}| > R$. Therefore, $f$ has compact support in both variables. Thus $f \in \mathcal{C}_c^\infty(\mathbb{R}^n_+)$.

It remains to show that \(R^k f=u\). By \eqref{scaled_proj_data}, it is enough
to prove
\[
v_n^{k+1}R^k f(a,v)=w(\bar a,\bar p),
\qquad
\bar p=\frac{\bar v}{v_n}, \quad v = (\bar v, v_n).
\]

Letting \(x_n=v_n r\), we have
\[
v_n^{k+1}R^k f(a,v)
=
\int_0^\infty f(\bar a+x_n\bar p,x_n)x_n^k\,dx_n.
\]
Using \eqref{f_def_from_h}, we obtain
\[
v_n^{k+1} R^k f(a,v) = \int_0^\infty x_n^{-2} h\left(-\frac{1}{x_n}, \frac{\bar{a}}{x_n} + \bar{p}\right) dx_n.
\]
Next we change variables by letting $t = -1/x_n$, which implies $dx_n = x_n^2 dt$, and thus the integral becomes
\begin{align}\label{h_integral}
v_n^{k+1} R^k f(a,v) = \int_{-\infty}^0 h(t, \bar{p} - t\bar{a}) \, dt = \int_{\mathbb{R}} h(t, \bar{p} - t\bar{a}) \, dt,
\end{align}
since $h$ vanishes for $t \ge 0$ by condition \emph{(iii)}.

By definition, $h$ is the full inverse Fourier transform of $H$:
\[
h(t, \bar{y}) = \frac{1}{(2\pi)^n} \int_{\mathbb{R}^{n-1}} \int_{\mathbb{R}} e^{i(t\sigma + \bar{y}\cdot\bar{\xi})} H(\sigma, \bar{\xi}) \, d\sigma \, d\bar{\xi}.
\]
Substituting $\bar{y} = \bar{p} - t\bar{a}$ gives
\begin{align*}
h(t, \bar{p} - t\bar{a}) &= \frac{1}{(2\pi)^n} \int_{\mathbb{R}^{n-1}} \int_{\mathbb{R}} e^{i(t\sigma + (\bar{p} - t\bar{a})\cdot\bar{\xi})} H(\sigma, \bar{\xi}) \, d\sigma \, d\bar{\xi} \\
&= \frac{1}{(2\pi)^n} \int_{\mathbb{R}^{n-1}} \int_{\mathbb{R}} e^{i\bar{p}\cdot\bar{\xi}} e^{it(\sigma - \bar{a}\cdot\bar{\xi})} H(\sigma, \bar{\xi}) \, d\sigma \, d\bar{\xi}.
\end{align*}
We now insert this expression back into \eqref{h_integral}. Because $h$ is smooth and compactly supported, its Fourier transform $H$ is a Schwartz function. Thus, in the sense of tempered distributions, we have
\[
v_n^{k+1} R^k f(a,v) = \frac{1}{(2\pi)^n} \int_{\mathbb{R}^{n-1}} e^{i\bar{p}\cdot\bar{\xi}} \int_{\mathbb{R}} H(\sigma, \bar{\xi}) \left[ \int_{\mathbb{R}} e^{it(\sigma - \bar{a}\cdot\bar{\xi})} dt \right] d\sigma \, d\bar{\xi}.
\]
The innermost integral over $t$ evaluates to the Dirac delta distribution, $2\pi \delta(\sigma - \bar{a}\cdot\bar{\xi})$. Evaluating the $\sigma$-integral against this delta distribution yields
\[
v_n^{k+1} R^k f(a,v) = \frac{1}{(2\pi)^{n-1}} \int_{\mathbb{R}^{n-1}} e^{i\bar{p}\cdot\bar{\xi}} H(\bar{a}\cdot\bar{\xi}, \bar{\xi}) \, d\bar{\xi}.
\]
By condition \emph{(ii)}, we know that $H(\bar{a}\cdot\bar{\xi}, \bar{\xi}) = W(\bar{a}, \bar{\xi})$. We thus obtain
\[
v_n^{k+1} R^k f(a,v) = \frac{1}{(2\pi)^{n-1}} \int_{\mathbb{R}^{n-1}} e^{i\bar{p}\cdot\bar{\xi}} W(\bar{a}, \bar{\xi}) \, d\bar{\xi} = w(\bar{a}, \bar{p}),
\]
since $w$ is the inverse Fourier transform of $W$ in $\bar \xi$. Because $v_n > 0$, we conclude that $R^k f(a,v) = u(\bar{a}, v)$, which completes the proof of sufficiency. 

Finally, we show that the two statements in condition \emph{(ii)} are equivalent. First, assume $W(\bar{a}, \bar{\xi}) = H(\sigma, \bar{\xi})$ where $\sigma = \bar{a} \cdot \bar{\xi}$. Direct computation yields that 
\[
\bar{b} \cdot \nabla_{\bar{a}} W(\bar{a},\bar{\xi}) = \bar{b} \cdot \nabla_{\bar{a}} H(\sigma,\bar{\xi}) 
= \bar{b} \cdot \left( \bar{\xi}\frac{\partial H}{\partial \sigma}  \right) 
= (\bar{b} \cdot \bar{\xi}) \frac{\partial H}{\partial \sigma}  = 0,
\]
for all $\bar{b}\in \mathbb{R}^{n-1}$ such that  $\bar{b} \cdot \bar{\xi} = 0$.

Conversely, assume that
$\bar b\cdot \nabla_{\bar a}W(\bar a,\bar\xi)=0$
for all $\bar b\in\mathbb R^{n-1}$
with $\bar b\cdot\bar\xi=0, \; \bar\xi=0.$
Let \(\bar a_1,\bar a_2\in\mathbb R^{n-1}\)
satisfy
$\bar a_1\cdot\bar\xi=\bar a_2\cdot\bar\xi.$
Then \((\bar a_1-\bar a_2)\cdot\bar\xi=0\). Hence, by the Fundamental Theorem of Calculus and the assumed differential condition, we obtain
\[
W(\bar{a}_1, \bar{\xi}) - W(\bar{a}_2, \bar{\xi}) = \int_{0}^{1} (\bar{a}_1 - \bar{a}_2) \cdot \nabla_{\bar{a}} W(\bar{a}_2 + t(\bar{a}_1 - \bar{a}_2), \bar{\xi}) \, dt=0,
\]
which implies 
\[
W(\bar{a}_1, \bar{\xi}) = W(\bar{a}_2, \bar{\xi}).
\]
Since $W$ takes the same value for all $\bar{a}$ having the same dot product with $\bar{\xi}$, it follows that $W$ is a function of that dot product $\bar{a}\cdot\bar{\xi}$ and we can write $W(\bar{a}, \bar{\xi}) = H(\bar{a} \cdot \bar{\xi}, \bar{\xi})$.
\end{proof}

As a corollary, we provide a moment-based range characterization of the $k$-weighted divergent beam transform, extending the work of \cite{ClackdoyleDesbat2013} from the three-dimensional case to general dimensions $n \ge 2$. Moreover, while they considered the unweighted divergent beam transform, our work accommodates more general weighted versions; setting $k=0$ in our result directly recovers theirs. We further note that the conditions presented in \cite[Theorem 5]{ClackdoyleDesbat2013} do not constitute a full set of consistency conditions, as they omit condition \emph{(iii)}, which is crucial to guarantee that $f \in \mathcal{C}^\infty_c(\mathbb{R}^n_+)$.

\begin{corollary}[Moment-based range characterization of the $k$-weighted divergent beam transform]\label{cor:moment-ray}
Let $n\ge 2$ and $k\in \mathbb{N}_0$. Let 
\[
u\in \mathcal{C}^\infty(\mathbb{R}^{n-1}\times \mathbb{S}^{n-1}_+).
\]
Then, $u$ belongs to the range of the $k$-weighted divergent beam transform, that is, there exist $f\in \mathcal{C}_c^\infty(\mathbb{R}^n_+)$ such that $u(\bar{a}, v) = R^k f(a, v)$ for $a = (\bar{a}, 0)$ with $\bar{a}\in \mathbb{R}^{n-1}$, if and only if the scaled projective data 
\[
w(\bar{a},\bar{p}) = v_n^{k+1} u (\bar{a}, v), \quad \bar{p} = \frac{\bar{v}}{v_n} \in \mathbb{R}^{n-1},
\]
satisfies the following conditions:

\begin{enumerate}
\item \textbf{Compact Support:} For every fixed $\bar{a} \in \mathbb{R}^{n-1}$, the function $w(\bar{a}, \cdot)$ has compact support.

\item \textbf{Moment Condition:} For every $m \in \mathbb{N}_0$, the $m$-th order moment
\[
J_m(\bar{a}, \bar{\xi}) := \int_{\mathbb{R}^{n-1}} (\bar{p}\cdot \bar{\xi})^m w(\bar{a}, \bar{p}) \,d\bar{p},
\]
is a homogeneous polynomial of degree $m$ in $\bar{\xi}$ and $\bar{a}\cdot\bar{\xi}$. That is, there exist homogeneous polynomials $c_{m,j}(\bar{\xi})$ of degree $m-j$ in $\bar{\xi}$ for $j = 0,\cdots, m$, such that
\[
J_m(\bar{a}, \bar{\xi}) = \sum_{j=0}^m c_{m,j}(\bar{\xi}) (\bar{a}\cdot\bar{\xi})^j.
\]

\item \textbf{Paley-Wiener Growth:} 
Define the Taylor coefficients $Q_j(\bar{\xi})$ by summing over the moment coefficients $c_{m,j}(\bar{\xi})$ for $m\ge j$:
\[
Q_j(\bar{\xi}) := \sum_{m=j}^\infty \frac{(-i)^m}{m!} c_{m,j}(\bar{\xi}).
\]
The formal power series
\[
H(\sigma, \bar{\xi}) := \sum_{j=0}^\infty Q_j(\bar{\xi}) \sigma^j
\]
extends to an entire function on $\mathbb{C} \times \mathbb{C}^{n-1}$. Furthermore, there exist constants $0 < m_0 < M_0 < \infty$ and $R > 0$ such that for every integer $N \ge 0$, there is a constant $C_N > 0$ satisfying the bound:
\begin{align}\label{PW_bound}
|H(\sigma, \bar{\xi})| \le C_N (1 + |\sigma| + |\bar{\xi}|)^{-N} \exp \left( \max_{t \in \{-m_0, -M_0\}} t \big( \operatorname{Im}\sigma - R |\operatorname{Im}\bar{\xi}| \big) \right).
\end{align}
\end{enumerate}
\end{corollary}

\begin{remark}
The conditions in this corollary are equivalent to those in
\Cref{t:rangeRayPlane}, but they give a moment-based formulation of the same range characterization.

The moment condition (ii) is obtained by expanding the Fourier transform
$W(\bar a,\bar\xi)$ in its Taylor series with respect to $\bar\xi$. It expresses,
term-by-term, the same dependence on $\bar a$ through the scalar quantity
$\bar a\cdot\bar\xi$ that appears in the factorization
\[
W(\bar a,\bar\xi)=H(\bar a\cdot\bar\xi,\bar\xi).
\]
Condition (iii) is the Paley--Wiener form of the support and smoothness
condition in \Cref{t:rangeRayPlane}. Instead of imposing support restrictions on
the inverse Fourier transform $h(t,\bar p)$, it imposes the corresponding entire
extension and exponential growth bounds on $H(\sigma,\bar\xi)$.
\end{remark}

\begin{proof}[Proof of Corollary~\ref{cor:moment-ray}]
We show that conditions \emph{(ii)} and \emph{(iii)} in the corollary together
give an equivalent reformulation of the second and third conditions in
\Cref{t:rangeRayPlane}. Condition \emph{(i)} is identical in both statements.

Let $W(\bar a,\bar\xi)$ be the Fourier transform of $w(\bar a,\bar p)$ with
respect to $\bar p$. Since $w(\bar a,\cdot)$ has compact support for each fixed
$\bar a$, the function $W(\bar a,\bar\xi)$ is entire in $\bar\xi$ and admits the
Taylor expansion
\[
W(\bar a,\bar\xi)
=
\int_{\mathbb R^{n-1}}
e^{-i\bar p\cdot\bar\xi}
w(\bar a,\bar p)\,d\bar p
=
\int_{\mathbb R^{n-1}}
\sum_{m=0}^\infty
\frac{(-i)^m}{m!}
(\bar p\cdot\bar\xi)^m
w(\bar a,\bar p)\,d\bar p
=
\sum_{m=0}^\infty
\frac{(-i)^m}{m!}
J_m(\bar a,\bar\xi).
\]

First assume that $u$ satisfies the conditions of \Cref{t:rangeRayPlane}. Then
there exists a function $H(\sigma,\bar\xi)$ such that
\[
W(\bar a,\bar\xi)=H(\bar a\cdot\bar\xi,\bar\xi),
\]
and the inverse Fourier transform of $H$ is a smooth function supported in
\[
K
=
\left\{
(t,\bar p)\in\mathbb R\times\mathbb R^{n-1}
:
t\in[-M_0,-m_0],
\ |\bar p|\le -Rt
\right\}.
\]
By the Paley-Wiener theorem \cite[Theorem 7.3.1]{hormander2003analysis}, $H$ extends to an entire function on
$\mathbb C\times\mathbb C^{n-1}$ and satisfies the corresponding
Paley-Wiener growth estimate.

We now derive the moment condition. Since $H$ is entire, we can write its Taylor
expansion in the variable $\sigma$ as
\[
H(\sigma,\bar\xi)
=
\sum_{j=0}^\infty A_j(\bar\xi)\sigma^j,
\]
where each $A_j$ is entire in $\bar\xi$. Decomposing each $A_j$ into its Taylor components,
\[
A_j(\bar\xi)
=
\sum_{\ell=0}^\infty A_{j,\ell}(\bar\xi),
\]
where $A_{j,\ell}$ is homogeneous of degree $\ell$ in $\bar\xi$, and substituting
$\sigma=\bar a\cdot\bar\xi$ gives
\[
W(\bar a,\bar\xi)
=
\sum_{j=0}^\infty
\sum_{\ell=0}^\infty
A_{j,\ell}(\bar\xi)
(\bar a\cdot\bar\xi)^j.
\]
The terms homogeneous of degree $m$ in $\bar\xi$ are precisely
\[
\sum_{j=0}^m
A_{j,m-j}(\bar\xi)(\bar a\cdot\bar\xi)^j.
\]
On the other hand, the homogeneous term of degree $m$ in the Taylor expansion of
$W$ is
\[
\frac{(-i)^m}{m!}J_m(\bar a,\bar\xi).
\]
Therefore,
\[
J_m(\bar a,\bar\xi)
=
\sum_{j=0}^m
c_{m,j}(\bar\xi)(\bar a\cdot\bar\xi)^j,
\qquad
c_{m,j}(\bar\xi)
=
\frac{m!}{(-i)^m}A_{j,m-j}(\bar\xi).
\]
Since $A_{j,m-j}$ is homogeneous of degree $m-j$, each
$c_{m,j}(\bar\xi)$ is homogeneous of degree $m-j$. This proves the moment
condition.

Moreover,
\[
\sum_{m=j}^\infty
\frac{(-i)^m}{m!}c_{m,j}(\bar\xi)
=
\sum_{m=j}^\infty A_{j,m-j}(\bar\xi)
=
A_j(\bar\xi).
\]
Thus the function constructed in condition \emph{(iii)} of the corollary agrees
with the same function $H$ appearing in \Cref{t:rangeRayPlane}. The
Paley-Wiener growth estimate is therefore exactly the Fourier-space form of the
support condition for the inverse Fourier transform of $H$.

Conversely, assume that conditions \emph{(i)}--\emph{(iii)} of the corollary
hold. Define
\[
H(\sigma,\bar\xi)
=
\sum_{j=0}^\infty Q_j(\bar\xi)\sigma^j,
\qquad
Q_j(\bar\xi)
=
\sum_{m=j}^\infty
\frac{(-i)^m}{m!}c_{m,j}(\bar\xi).
\]
By condition \emph{(iii)}, this formal series extends to an entire function on
$\mathbb C\times\mathbb C^{n-1}$ satisfying the Paley-Wiener estimate \eqref{PW_bound}.

We claim that
\[
W(\bar a,\bar\xi)=H(\bar a\cdot\bar\xi,\bar\xi).
\]
For fixed $\bar a$, both sides are entire functions of $\bar\xi$. The Taylor
series of $W$ at $\bar\xi=0$ is
\[
W(\bar a,\bar\xi)
=
\sum_{m=0}^\infty
\frac{(-i)^m}{m!}
J_m(\bar a,\bar\xi).
\]
Using the moment condition, this becomes
\[
W(\bar a,\bar\xi)
=
\sum_{m=0}^\infty
\frac{(-i)^m}{m!}
\sum_{j=0}^m
c_{m,j}(\bar\xi)(\bar a\cdot\bar\xi)^j.
\]
On the other hand, the Taylor expansion of
$H(\bar a\cdot\bar\xi,\bar\xi)$ at $\bar\xi=0$ is, by construction,
\[
H(\bar a\cdot\bar\xi,\bar\xi)
=
\sum_{j=0}^\infty
Q_j(\bar\xi)(\bar a\cdot\bar\xi)^j
=
\sum_{j=0}^\infty
\sum_{m=j}^\infty
\frac{(-i)^m}{m!}
c_{m,j}(\bar\xi)(\bar a\cdot\bar\xi)^j.
\]
The homogeneous term of degree $m$ in $\bar\xi$ is
\[
\frac{(-i)^m}{m!}
\sum_{j=0}^m
c_{m,j}(\bar\xi)(\bar a\cdot\bar\xi)^j
=
\frac{(-i)^m}{m!}
J_m(\bar a,\bar\xi).
\]
Thus $W(\bar a,\bar\xi)$ and
$H(\bar a\cdot\bar\xi,\bar\xi)$ have the same Taylor expansion at
$\bar\xi=0$. Since both are entire in $\bar\xi$, they are identical. Hence the
factorization condition in \Cref{t:rangeRayPlane} holds.

It remains to identify the support condition. By the Paley-Wiener theorem, the stated estimate for $H$ in \eqref{PW_bound} is equivalent to the statement that its inverse Fourier transform
$h(t,\bar p)$ is a smooth function supported in the compact convex set
\[
K
=
\left\{
(t,\bar p)\in\mathbb R\times\mathbb R^{n-1}
:
t\in[-M_0,-m_0],
\ |\bar p|\le -Rt
\right\}.
\]
Indeed, the support function of $K$ is
\begin{align*}
H_K(\operatorname{Im}\sigma,\operatorname{Im}\bar\xi)
&=
\sup_{(t,\bar p)\in K}
\big(
t\operatorname{Im}\sigma
+
\bar p\cdot\operatorname{Im}\bar\xi
\big)
\\
&=
\sup_{t\in[-M_0,-m_0]}
\left[
t\operatorname{Im}\sigma
+
\sup_{|\bar p|\le -Rt}
\bar p\cdot\operatorname{Im}\bar\xi
\right]
\\
&=
\sup_{t\in[-M_0,-m_0]}
t\big(
\operatorname{Im}\sigma
-
R|\operatorname{Im}\bar\xi|
\big)
\\
&=
\max_{t\in\{-m_0,-M_0\}}
t\big(
\operatorname{Im}\sigma
-
R|\operatorname{Im}\bar\xi|
\big).
\end{align*}
Therefore condition \emph{(iii)} of the corollary is exactly the
Paley-Wiener formulation of the support and smoothness condition in
\Cref{t:rangeRayPlane}.

Thus conditions \emph{(ii)} and \emph{(iii)} of the corollary together are
equivalent to the factorization and support conditions in
\Cref{t:rangeRayPlane}. Since condition \emph{(i)} is the same in both
statements, the corollary follows from \Cref{t:rangeRayPlane}.
\end{proof}
 
The equivalence between \Cref{t:rangeRayPlane} and Corollary~\ref{cor:moment-ray} should
be understood as an equivalence of the full sets of conditions, not as a
pairwise equivalence of the individual conditions.
The reason is that the factorization condition in \Cref{t:rangeRayPlane},
\[
W(\bar a,\bar\xi)=H(\bar a\cdot\bar\xi,\bar\xi),
\]
and the support condition on the inverse Fourier transform of \(H\) are coupled
through the same function \(H\). Similarly, in Corollary~\ref{cor:moment-ray}, the moment
condition produces the coefficients used to define
\[
H(\sigma,\bar\xi)=\sum_{j=0}^\infty Q_j(\bar\xi)\sigma^j,
\]
and the Paley--Wiener condition is imposed on this same function \(H\).

Thus the moment condition alone is not equivalent to the PDE/factorization
condition alone. The PDE condition only describes how \(W\) depends on \(\bar a\),
namely through the scalar quantity \(\bar a\cdot\bar\xi\). By itself, it does
not guarantee that the Taylor coefficients of \(W\) have the polynomial moment
structure required in Corollary~\ref{cor:moment-ray}. Conversely, the moment condition
alone gives a coefficient-level constraint, but without the Paley--Wiener
condition it does not guarantee that the resulting formal series defines an
entire function \(H\) whose inverse Fourier transform has the required support
and smoothness.

Therefore, conditions \emph{(ii)} and \emph{(iii)} in the corollary should be
viewed together: they jointly reformulate conditions \emph{(ii)} and
\emph{(iii)} in \Cref{t:rangeRayPlane}. The individual conditions are not
freely interchangeable unless one also verifies that the same function \(H\) is
being used in both formulations.

\subsection{\texorpdfstring{\boldmath Range of the $k$-weighted Compton transform for infinite planar detectors}{}}
We now characterize the range of the $k$-weighted Compton transform in the case of planar detector geometry by combining the range conditions for $S$ and $R^k$, stated in Theorems~\ref{t:rangeSST} and~\ref{t:rangeRayPlane}, respectively.

\begin{theorem}[Range characterization of the $k$-weighted Compton transform]
\label{t:rangeCompton}
Let $n \ge 3$ and $k \in \mathbb{N}_0$. Let $$g \in \mathcal{C}^\infty(\mathbb{R}^{n-1} \times \mathbb{S}^{n-1} \times (-1,1)).$$
Then, $g$ belongs to the range of the $k$-weighted Compton transform, that is, $g = C^k f$, for some $f\in \mathcal{C}_c^\infty(\mathbb{R}^n_+)$, if and only if $g$ satisfies the following conditions:

\begin{enumerate}
\item \textbf{Evenness:} For every $\bar{a} \in \mathbb{R}^{n-1}$, 
\[
g(\bar{a}, -\beta, -s) = g(\bar{a}, \beta, s), \qquad (\beta, s) \in \mathbb{S}^{n-1} \times (-1,1).
\]
\item \textbf{PDE Condition:} For every $\bar{a} \in \mathbb{R}^{n-1}$, 
\[
\Big[(1-s^2)\partial_s^2 + (n-3)s\partial_s + \frac{n-2}{1-s^2} - \Delta_{\mathbb{S}^{n-1}}\Big] g(\bar{a}, \beta, s) = 0, \qquad (\beta, s) \in \mathbb{S}^{n-1} \times (-1,1).
\]
\item \textbf{Limit Condition:} The limit
\[
u(\bar{a}, v) := \lim_{s \to 1} \frac{g(\bar{a}, v, s)}{|\mathbb{S}^{n-2}|(1-s^2)^{(n-2)/2}}
\]
exists for every $(\bar{a}, v) \in \mathbb{R}^{n-1} \times \mathbb{S}^{n-1}$ and defines a smooth function $u \in \mathcal{C}^\infty(\mathbb{R}^{n-1} \times \mathbb{S}^{n-1})$. Furthermore, $u(\bar{a}, v) = 0$ for all directions $v = (\bar{v}, v_n)$ where $v_n \le 0$.
\item \textbf{Projective Data Conditions:} For directions $v_n > 0$, define the projective variable $\bar{p} = \bar{v}/v_n \in \mathbb{R}^{n-1}$ and the scaled projective data $w(\bar{a}, \bar{p}) := v_n^{k+1}u(\bar{a}, v)$. Let $W(\bar{a}, \bar{\xi})$ be the Fourier transform of $w$ with respect to $\bar{p}$. Then $w$  and $W$ satisfy:
\begin{enumerate}
\item[(a)] \textbf{Compact Support:} For every fixed $\bar{a} \in \mathbb{R}^{n-1}$, the function $w(\bar{a}, \cdot)$ has compact support.
\item[(b)] \textbf{PDE Condition:} There exists a function $H(\sigma, \bar{\xi})$ such that 
$$W(\bar{a}, \bar{\xi}) = H(\bar{a}\cdot\bar{\xi}, \bar{\xi}).$$
\item[(c)] \textbf{Support and Smoothness:} The inverse Fourier transform of $H$ with respect to both variables, defined as
\[
h(t, \bar{p}) := \frac{1}{(2\pi)^n} \int_{\mathbb{R}^{n-1}} \int_{\mathbb{R}} e^{i(t\sigma + \bar{p}\cdot\bar{\xi})} H(\sigma, \bar{\xi}) \, d\sigma \, d\bar{\xi},
\]
is a smooth function ($h \in \mathcal{C}^\infty(\mathbb{R} \times \mathbb{R}^{n-1})$). Furthermore, there exist constants $0 < m_0 < M_0 < \infty$ and $R > 0$ such that the support of $h$ satisfies
\[
\supp h \subset \left\{ (t, \bar{p}) \in \mathbb{R} \times \mathbb{R}^{n-1} : t \in [-M_0, -m_0] \text{ and } |\bar{p}| \le -R t \right\}.
\]
\end{enumerate}
\end{enumerate}
\end{theorem}

\begin{proof}
We first prove necessity. Let $f\in \mathcal{C}_c^\infty(\mathbb{R}^n)$ with $\operatorname{supp} f \subset \mathbb{R}^n_+$. We will show that the function
\[
g(\bar{a}, \beta, s) := C^k f(\bar{a}, \beta, s) = S(R_{a}^k f)(\beta, s), \quad a = (\bar{a}, 0),
\]
satisfies the conditions of the theorem.

Let
\[
u(\bar{a}, v) := R^k f(a, v).
\]
Then, 
\[
g(\bar{a}, \beta, s) = S(u(\bar{a}, \cdot))(\beta, s).
\]
 
By \Cref{t:rangeSST}, the range characterization of the spherical section transform, conditions \emph{(i)} and \emph{(ii)} hold immediately, and the limit in condition \emph{(iii)} recovers $u(\bar{a}, v)$:
\[
\lim_{s \to 1} \frac{g(\bar{a}, v, s)}{|\mathbb{S}^{n-2}|(1-s^2)^{(n-2)/2}} = u(\bar{a}, v).
\]

Since $\operatorname{supp} f \subset \mathbb{R}^n_+$, any ray originating at $a = (\bar{a}, 0)$ pointing in a direction $v = (\bar{v}, v_n)$ with $v_n \le 0$ will strictly miss the support of $f$. Thus, the integral $u(\bar{a}, v) = R^k f(a, v) = 0$ for all $v_n \le 0$. This fully satisfies condition \emph{(iii)}.

Lastly, because $u(\bar{a}, v) = R^k f(a, v)$ is generated by a function supported in $\mathbb{R}^n_+$, the range characterization of the $k$-weighted divergent beam transform guarantees that the scaled projective data $w$ and its Fourier transform $W$ satisfy the Compact Support, PDE, and Paley-Wiener constraints. This establishes condition \emph{(iv)}.

We now prove sufficiency. Suppose $g$ satisfies conditions \emph{(i)} through \emph{(iv)}. We will construct a function $f$ with $\operatorname{supp} f \subset \mathbb{R}^n_+$ such that $g = C^k f$.

First, we define $u(\bar{a}, v)$ using the limit defined in condition \emph{(iii)}:
\[
u(\bar{a}, v) := \lim_{s \to 1} \frac{g(\bar{a}, v, s)}{|\mathbb{S}^{n-2}|(1-s^2)^{(n-2)/2}}.
\]
By condition \emph{(iii)}, $u(\bar{a}, v) = 0$ whenever $v_n \le 0$, meaning the non-zero data is isolated to the upper hemisphere $v \in \mathbb{S}^{n-1}_+$. Furthermore, condition \emph{(iv)} guarantees that the projective data $w(\bar{a}, \bar{p}) := v_n^{k+1}u(\bar{a}, v)$ satisfies the required Compact Support, PDE, and Paley-Wiener bounds.

Consequently, by \Cref{t:rangeRayPlane}, the range characterization of the $k$-weighted divergent beam transform, there exists a smooth function $f$ with $\operatorname{supp} f \subset \mathbb{R}^n_+$ such that
\[
u(\bar{a}, v) = R^k f(a, v),
\]
for all $a = (\bar{a}, 0)$.

Next we fix an arbitrary $\bar{a} \in \mathbb{R}^{n-1}$. By our definition of $u$ from the limit, and applying the symmetry $g(\bar{a}, -\beta, -s) = g(\bar{a}, \beta, s)$ from condition \emph{(i)}, the opposite limit at $s \to -1$ evaluates to $u(\bar{a}, -v)$. Combined with the PDE constraint in condition \emph{(ii)}, all prerequisites for the spherical section transform range characterization are met for this fixed source $\bar{a}$.

Thus, we conclude that $g(\bar{a}, \beta, s)$ is the spherical section transform of $u(\bar{a}, \cdot)$:
\[
g(\bar{a}, \beta, s) = S(u(\bar{a}, \cdot))(\beta, s).
\]
Substituting our relation $u(\bar{a}, v) = R^k f(a, v)$ into this equation gives:
\[
g(\bar{a}, \beta, s) = S(R^k f)(\beta, s) = C^k f(\bar{a}, \beta, s).
\]
Since $\bar{a}\in\mathbb{R}^{n-1}$ was arbitrary, this shows that $g$ lies in the
range of the $k$-weighted Compton transform, completing the proof.
\end{proof}

\section{Conclusions}\label{s:conclusions}
We have obtained complete data consistency conditions for weighted cone integral transforms
in two types of vertex geometries that arise in integral geometry and Compton
camera imaging. The results show that the range problem is governed not only by
the conical integration surface and the weight, but also by the position of the
vertex set relative to the support of the unknown function.

A common feature in both cases is the factorization of the transform into a
$k$-weighted divergent beam transform followed by the spherical section
transform. This factorization separates the angular part of the problem from the
divergent beam component. The range conditions for the spherical section
transform are common to both geometries, while the range conditions for the
divergent beam transform depend on the underlying vertex geometry.

When the vertex set is bounded, convex, and contains the support of the unknown
function, the range of the $k$-weighted divergent beam transform is characterized by a higher-order transport boundary-value problem. This leads to the range characterization of the
$k$-weighted conical Radon transform. In contrast, for the Compton transform we
studied the planar detector geometry, where the vertices lie on $\mathbb{R}^{n-1}$ and are disjoint from the support of the radiation density. In this case, the range of the $k$-weighted divergent beam transform is described by cone-beam type consistency conditions generalizing those of Clackdoyle and Desbat \cite{ClackdoyleDesbat2013}, yielding the range characterization of the $k$-weighted Compton transform.

The present work treats the planar detector geometry for the Compton transform.
An interesting direction for future work is to extend the analysis to other
vertex geometries, such as vertices restricted to curved detector surfaces or
lower-dimensional detector bin setups.

\section*{Acknowledgements}
The work of the first author was supported in part by NSF DMS grant 2206279. 


\bibliographystyle{siam}
\bibliography{references}

\end{document}